\documentclass{birkmult}
\usepackage{amssymb}
\usepackage{graphicx,psfrag,color,pstricks,pst-grad}
 \newtheorem{thm}{Theorem}[section]

 \newtheorem{prop}[thm]{Proposition}
 \theoremstyle{definition}
 \newtheorem{defn}[thm]{Definition}
 \theoremstyle{remark}

 \numberwithin{equation}{section}
\begin{document}
%
%
%
%
%
%
%
\title[Linear stability of the $n$-gon relative equilibria of the $(1+n)$-body problem]
 {Linear stability of the $n$-gon relative equilibria of the $(1+n)$-body problem}
\author[XU]{Xingbo Xu $^{1,2,3}$ \\
\footnotesize 1. Purple Mountain Observatory, Chinese Academy of Sciences, Nanjing 210008, China\\
2. Graduate University of Chinese Academy of Sciences, Beijing 100039, China \\
3. IMCCE, CNRS-UMR 8028, Observatoire de Paris, 77 Avenue
Denfert-Rochereau, Paris 75014, France }

\address{No. 2,
West Beijing Road\\
210008, Nanjing\\
China}

\email{xbxu@pmo.ac.cn, }

\thanks{This work is supported by the National Natural Science Foundation of China under grant No.10833001. This paper has been published in Qual.Theory Dyn. Syst. (2013) 12: 255-271,
DOI 10.1007/s12346-012-0089-6, and is available at 
http://link.springer.com/article/10.1007/s12346-012-0089-6 }


\date{\small \today}

\begin{abstract}
We consider the linear stabilities of the regular $n$-gon relative
equilibria of the $(1+n)$-body problem. It is shown that there exist
at most two kinds of infinitesimal bodies arranged alternatively at
the vertices of a regular $n$-gon when $n$ is even, and only one set
of identical infinitesimal bodies when $n$ is odd. In the case of
$n$ even, the regular $n$-gon relative equilibrium is shown to be
linearly stable when $n\geqslant 14$. In each case of $n=8,10$ and
$12$, linear stability can also be preserved if the ratio of two
kinds of masses belongs to an open interval. When $n$ is odd, the
related conclusion on the linear stability is recalled.
\end{abstract}

\maketitle
\section{Introduction}

The famous essay Maxwell(1859)\cite{Maxwell} discussed the stability of the rings of Saturn
by mathematical modeling. Under the hypothesis that the motion of the rings is uniform and should be stable,
he abandoned the possibilities that the rings are connected solid, or continuous liquid. He
concluded that the rings should be composed of countless discrete particles (\cite{Maxwell}, p.66)
after an important intermediate study. In that study, he conceived a very nearly circular and uniform
ring of satellites with equal masses, and found that the ring is linearly stable provided the ratio
of Saturn's mass to the total mass of $n$ satellites is greater than $0.4352 n^{2}$ (\cite{Maxwell}, p.25 or p.59).
This essay has aroused attentions of many aspects. One can also refer to a French monument
written by Tisserand\cite{Tisserand}.

Pendse (1935)\cite{Pendse} pointed out that Maxwell's ring model could not be applied to
the cases when $n$ is smaller than seven, because Maxwell assumed tacitly the center of primary to be
the center of masses. Scheeres \& Vinh (1991)\cite{Scheeres} accepted Pendse's argument, and
analyzed the characteristic equations very attentively. They pointed out a miscalculation of Pendse, and concluded that
the upper bound of the ratio of the total mass of $n$ satellites to the dominant mass for stability
can be represented by an asymptotic series of $n$ if $n\geq 7$, and is about $1/(0.4352 n^{2})$
when $n$ is large. Similar results are also achieved by
Roberts (2000)\cite{Roberts00} and Vanderbei \& Kolemen (2007)\cite{Vanderbei}.

Consider the planar $N$-body problem with $N=n+1$ in a uniformly rotating coordinate system with the origin at
the center of masses. Suppose one mass is large and the other $n$ masses are very small.
As the order of magnitude of the $n$ small masses tends to zero, the limiting case of a relative equilibrium of this problem,
is called a \emph{relative equilibrium} of the $(1+n)$-body problem, which was defined by G. R. Hall
in an unpublished paper \cite{Hall}. The relationship between the relative equilibria of
$(1+n)$-body problem and those of the $N$-body problem with $N=n+1$ has been revealed by Moeckel (1994)\cite{Moeckel}, and
will be mentioned in section 2.

If all the infinitesimal masses are equal, Casasayas et al. (1994)\cite{Casasaya}
showed that when $n> e^{73}$, the only stationary configuration is
Maxwell's ring configuration. When $n\leq 8$, other configurations exist, referring to Salo \& Yoder (1988)\cite{Salo}.
According to Salo \& Yoder's numerical quest, it seems that there is only one type of stationary configuration
when $n\geq 9$, and this conjecture is also supported by the numerical experiment of Carles Sim\'{o} (\cite{Cors04}, p.326).
When $2\leq n \leq 4$, the numerical results on the numbers of stationary configurations also concide
with analytical proofs, see Cors et al. (2004)\cite{Cors04}, Albouy \& Fu (2009)\cite{AlFu09}.
For the cases of $n=5,6,7,8,9$, they seem to be much more difficult to deal with,
as the case of $n=4$ had already been claimed by Hall to be difficult to handle (\cite{Hall}, p.12).

If the infinitesimal masses are arbitrarily given, Maxwell said that
\textquotedblleft we must calculate the disturbing forces due to any given displacement of the ring\textquotedblright
(\cite{Maxwell}, p.38), so it is difficult to determine the stability of a non-regular configuration.
Salo \& Yoder (1988) studied a special case of $n=3$ with a background of Saturn's coorbital satellites,
and Renner \& Sicardy (2004)\cite{Renner} studied the cases
of $n=3,4,5$. Recently, Corbera et al.(2011)\cite{Corbera11} researched into the case of $n=3$
and obtained two classes of new configurations generated by changing the infinitesimal masses. They also
calculated the number of configurations with a high precision, and found that the number varies from five to seven.

The new results gotten in this paper are especially about the existence and linear stability of
the regular $n$-gon relative equilibria in the $(1+n)$-body problem with two kinds of infinitesimal masses.
In section 2.1, we introduce Hall's potential function on the regular $n$-gon relative equilibria
of the $(1+n)$-body problem. Then the method for calculating eigenvalues of the Hessian matrix
of Hall's potential function is described in section 2.2 and 2.3.
In section 3, we show the existence of a set of positive infinitesimal masses
for the regular $n$-gon relative equilibria. In section 4, we show
the linear stabilities of regular $n$-gon relative equilibria with $n$ even.
The last section is the conclusion part, and we also give some vistas for future study.

\section{Regular $n$-gon relative equilibria}

There are two ways for the definition of a relative equilibrium
of the $(1+n)$-body problem. One way is to make a first approximation of the full system,
then write out the conditional equations for a relative equilibrium, for example, Salo, Yoder (1988).
The other way is to make a limiting case of the conditional equations for a
relative equilibrium of the full system, for example, Moeckel (1994).
We obtain the conditional equations for relative equilibria of the $(1+n)$-body problem by the second way.

\subsection{Relative equilibria}
Consider the Newtonian $N$-body problem with $N=n+1$ in a rotating coordinate frame
with the origin at the center of masses. Set the angular velocity of this frame as
a constant $\omega=1$. Suppose there is only one dominant mass and the other $n$ masses are very small.
Denote the dominant mass as $m_{0}$, and the other masses as $m_{i}$, $i=1,\cdots, n$.
Introduce the positions and conjugate momentums of the small masses as $x_{i}, y_{i}\in \mathbb{R}^{2}$,
$i=1,\cdots, n$, separately. The Hamiltonian function can be written,
 \begin{equation}\label{Ha0}
 \mathcal{H}=\sum_{i=0}^{n}\left[\frac{\|y_{i}\|^{2}}{2m_{i}}-x_{i}' J y_{i}\right]
-\sum_{i=0}^{n-1}\sum_{j=i+1}^{n}\frac{m_{i}m_{j}}{\|x_{j}-x_{i}\|} .
   \end{equation}
where
\begin{displaymath}
m_{0}=1, \quad m_{i}=\varepsilon \mu_{i}, \quad x_{0}=-\sum_{i=1}^{n}m_{i}x_{i}, \quad
y_{0}=-\sum_{i=1}^{n}y_{i},  \quad J=\left( \begin{array}{cc}
                                             0 & 1 \\
                                            -1 & 0
                                            \end{array}\right) .
\end{displaymath}
and the prime represents transposition.

Let's recall a theorem given by Moeckel (1994) in order to understand the relationship between
the relative equilibria of the $(1+n)$-body problem and those of the $N$-body problem with $N=n+1$.
\begin{thm}[Moeckel, 1994]\label{Moeckelthm}
 Let $x^{\varepsilon}$ be a family of relative equilibria of the $N$-body problem with $N=n+1$, $m_{0}=1$,
and $m_{i}=\varepsilon\mu_{i}$, $1\leq i \leq n$. Suppose that as $\varepsilon\rightarrow 0$, $x^{\varepsilon}$
converges to a non-degenerate relative equilibrium, $\bar{x}$, of the $(1+n)$-body problem. Then $x^{\varepsilon}$ is
non-degenerate for $\varepsilon$ sufficiently small. In this case, $x^{\varepsilon}$ is linearly stable
for $\varepsilon$ sufficiently small if and only if $\bar{x}$ is a local minimum of $\mathbf{V}$, which
is defined in \emph{(}\ref{VH}\emph{)} below.
\end{thm}
The following proposition can be found in both Hall (preprint) and Moeckel (1994).
\begin{prop}
 For $\bar{x}$ as given in the theorem above, we have $\bar{x}_{0}=0$, $|\bar{x}_{i}|=1$, $i=1,2,\cdots,n$.
\end{prop}

The conditional equations for a relative equilibrium of Hamiltonian system (\ref{Ha0}) can be written,
\begin{align} \label{RE1}
\dot{x}_{i}&=\frac{\partial \mathcal{H}}{\partial y_{i}}=\frac{y_{i}}{m_{i}}+Jx_{i}=0 ,
 \\
\dot{y}_{i}&=-\frac{\partial \mathcal{H}}{\partial x_{i}}=J y_{i}+ \frac{m_{i}m_{0}(x_{0}-x_{i})}{\|x_{0}-x_{i}\|^{3}}+
\sum_{k=1, \neq i}^{n}\frac{m_{i}m_{k}(x_{k}-x_{i})}{\|x_{k}-x_{i}\|^{3}}=0 \label{RE2} ,
\end{align}
where $i=1,2,\cdots, n$.

One gets $Jy_{i}=m_{i}x_{i}$ from (\ref{RE1}), and substitutes it into (\ref{RE2}).
Let $x_{i}=r_{i}$ $\left(\cos\vartheta_{i}, \sin \vartheta_{i}\right)'$,
(\ref{RE2}) make inner product with $(-\sin\vartheta_{i},\cos\vartheta_{i})'$,
and $\varepsilon\rightarrow 0$.
Then the conditional equations for a relative equilibrium of $(1+n)$-body problem can be obtained,
\begin{eqnarray}\label{ki1}
\sum_{k=1, k\neq i}^{n}\mu_{k}\sin(\vartheta_{k}-\vartheta_{i})\left( \frac{1}{r_{k,i}^{3}}-1 \right)=0 ,
\end{eqnarray}
where $r_{k,i}=2\sin\left( \frac{|\vartheta_{k}-\vartheta_{i}|}{2}\right)$, $ k\neq i$.

Integrate the left hand side of equation (\ref{ki1}) of $\vartheta_{i}$, yields
\begin{eqnarray}\label{VH}
 \mathbf{V} =\sum_{i=1}^{n-1}\sum_{k=i+1}^{n}\mu_{i}\mu_{k}\left(\frac{1}{r_{k,i}}
+\frac{r_{k,i}^{2}}{2}\right) ,
\end{eqnarray}
and function $\mathbf{V}$ is called \emph{Hall's potential function}.

Equations (\ref{ki1}) refers to a function below,
\begin{equation}\label{Ftheta}
\mathcal{F}(\varphi)=\sin \varphi \left( 1-\frac{1}{8|\sin^{3}\varphi/2|} \right),
\quad \varphi \in \{\varphi|\varphi\in \mathbb{R}, \varphi\neq 2k\pi,k\in \mathbb{Z} \} .
\end{equation}
The first order derivative of $\mathcal{F}(\varphi)$ is
 \begin{eqnarray}\label{fx}
\mathbf{f}(\varphi)=\mathcal{F}'(\varphi)
=\frac{1}{8|\sin \frac{\varphi}{2}|}\left( \frac{2}{\sin^{2}\frac{\varphi}{2}}-1\right)+\cos \varphi .
\end{eqnarray}
Note that functions $\mathbf{f}(\varphi)$ and $\mathcal{F}(\varphi)$ are useful for later analysis,
so let's recall a proposition, which can be found
in \cite{Hall},\cite{Casasaya},\cite{Cors04},\cite{AlFu09}.
\begin{prop}
Functions $\mathcal{F}(\varphi)$ and $\mathbf{f}(\varphi)$, with the field of definition
 $\{\varphi| \varphi\in \mathbb{R}, \varphi\neq 2k\pi,k\in \mathbb{Z} \}$, satisfy
 \begin{enumerate}
  \item $\mathcal{F}(-\varphi)=-\mathcal{F}(\varphi)$, $\quad \mathcal{F}(\varphi\pm2\pi)=\mathcal{F}(\varphi)$,
$\quad \mathcal{F}(\pi-\varphi)=-\mathcal{F}(\varphi-\pi)$,
  \item $\mathbf{f}(-\varphi)=\mathbf{f}(\varphi)$, $\quad \mathbf{f}(\pi-\varphi)=\mathbf{f}(\varphi-\pi)$,
$\quad \mathbf{f}(\varphi) \geq  \mathbf{f}(\pi)=-\frac{7}{8} $,
  \item $\exists \, c_{1}>0$, \text{such that} $\mathbf{f}(\varphi)>c_{1}/\varphi^{3}-1 $, \text{for} $\varphi\in(0,\pi)$,
  \item $\mathbf{f}''(\varphi)=\mathcal{F}'''(\varphi) > 0$, $\quad \forall \varphi\in(0,2\pi)$.
 \end{enumerate}
\end{prop}
These properties of the two functions are shown in the figure \ref{F1F2}.
\begin{figure}[b]
 \centering
 \includegraphics[width=70mm,bb=0 -0.250000000 525.250000 375.750000]{erhan.eps}
   \caption{ $\mathtt{F}(\varphi)$ represents function (\ref{Ftheta}) and $\mathtt{f}(\varphi)$
represents function (\ref{fx})
}
\label{F1F2}
\end{figure}

\subsection{Circulant matrix}
The linear stabilities of the relative equilibria of the $(1+n)$-body problem depend on
the eigenvalues of the Hessian matrix $\mathbf{V}_{\vartheta\vartheta}$.
When all the infinitesimal masses are equal, the matrix is a circulant matrix,
and Moeckel (1994) concluded that such regular $n$-gon relative equilibria
are linearly stable if and only if $n\geq 7$.
The definition of the circulant matrix is given below.
\begin{defn}
Let $j$ be a positive integer. Matrix
$\mathbf{A}$ of $j\times j$ is a \emph{circulant matrix}, if it satisfies
\begin{displaymath}
 A_{i,k}=A_{\{i+l\}_{j},\{k+l\}_{j}}, \qquad i, k, l\in \mathbb{N},
\end{displaymath}
where $A_{i,k}$ is an element of the matrix and $\{i+l\}_{j}$ represents $i+l$ modulo $j$.
\end{defn}

In a regular $n$-gon relative equilibrium, suppose one infinitesimal body lies
on the positive $x$ axis, and count them from this one by the anti clockwise direction.
The angle positions of the infinitesimal bodies are
\begin{equation}\label{angles}
\vartheta_{i}=\frac{2(i-1)\pi}{n}, \quad i=1,2,\cdots,n .
\end{equation}

Let's review the following proposition \ref{circu} about the circulant matrix, and this propostion
can be found in \cite{Gray}.

\begin{prop}\label{circu}
 Every circulant matrix $\mathbf{A}$ with rank $j$ has eigenvectors
\begin{equation*}
 \xi^{(l)}=\frac{1}{\sqrt{j}}\left(1,e^{-2\pi \mathrm{i}(l-1)/j},\cdots,e^{-2\pi \mathrm{i}(l-1)(j-1)/j}\right)',
\quad l=1,2,\cdots, j ,
\end{equation*}
where $\mathrm{i}=\sqrt{-1}$.

Denote the element in the $i$-th row and $k$-th column of matrix $\mathbf{A}$ as $A_{i,k}$.
The eigenvalues are,
\begin{equation*}\label{alpha}
 \alpha_{l}=\sum_{k=1}^{j}A_{1,k}e^{-2\pi \mathrm{i}(l-1)(k-1)/j}, \quad l=1,2,\cdots, j .
\end{equation*}
One has $\mathbf{A}=\mathbf{U}\mathbf{diag}\{\alpha_{l}\}\mathbf{U}^{-1}$,
where $\mathbf{U}$ has the eigenvectors as columns in order, and
$\mathbf{diag}\{\alpha_{l}\}$ is a diagonal matrix
of eigenvalues of matrix $\mathbf{A}$.

 Let $\mathbf{B}$ be another circulant $j\times j$ matrix, and denote its element in the $i$-th row and
$k$-th column as $B_{i,k}$. Its eigenvalues are
\begin{equation*}
 \beta_{l}=\sum_{k=1}^{j}B_{1,k}e^{-2\pi \mathrm{i}(l-1)(k-1)/j}, \quad l=1,2,\cdots, j ,
\end{equation*}
respectively. Then
\begin{enumerate}
 \item $\mathbf{A}$ and $\mathbf{B}$ commute and
 $ \mathbf{AB}=\mathbf{BA}=\mathbf{U}\mathbf{diag}\{\alpha_{l}\beta_{l}\}\mathbf{U}^{-1} $,
$\mathbf{AB}$ is also a circulant matrix.
\item $\mathbf{A}+\mathbf{B}$ is a circulant matrix and
 $  \mathbf{A}+\mathbf{B}=\mathbf{U}\mathbf{diag}\{\alpha_{l}+\beta_{l}\}\mathbf{U}^{-1} $.
\item  If $\alpha_{l}\neq 0$, $l=1,2,\cdots,j$, then $\mathbf{A}$ is nonsingular and
\begin{displaymath}
 \mathbf{A}^{-1}=\mathbf{U}\mathbf{diag}\{\alpha_{l}^{-1}\}\mathbf{U}^{-1}.
\end{displaymath}
\end{enumerate}
\end{prop}

However, if not all the infinitesimal masses are identical, $\mathbf{V}_{\vartheta\vartheta}$
will no longer be a circulant matrix.

\subsection{Eigenvalues of a block circulant matrix }

Define a block circulant matrix $\mathbf{S}$ as below,
\begin{equation} \label{block1}
\mathbf{S}=\left(\begin{array}{cc}
      \mathbf{A} & \mathbf{C} \\
      \mathbf{C}^{T} & \mathbf{B} \\
      \end{array}
\right) ,
\end{equation}
where $\mathbf{A}$, $\mathbf{B}$ and $\mathbf{C}$ are all circulant matrices of $j \times j$.
The upper $T$ represents transposition.

Denote $\mathbf{I}_{2j}$ as a $2j\times 2j$ unit matrix,
and $\mathbf{I}$ as a $j\times j$ unit matrix.
The characteristic polynomial of matrix $\mathbf{S}$ can be calculated
by applying some elementary transformations.
\begin{align}\label{tran1}
 \left|  \lambda\mathbf{I}_{2j}-\mathbf{S}\right|
&=\left|\begin{array}{cc}
       \lambda\mathbf{I}-\mathbf{A} & -\mathbf{C} \\
      -\mathbf{C}^{T} & \lambda\mathbf{I}-\mathbf{B} \\
       \end{array}\right|   \nonumber\\
&=\left|\left(\begin{array}{cc}
       \mathbf{I} & \mathbf{0} \\
-\mathbf{C}^{T}(\lambda\mathbf{I}-\mathbf{A})^{-1} & \mathbf{I} \\
       \end{array}\right)
\left(\begin{array}{cc}
      \lambda\mathbf{I}-\mathbf{A} & -\mathbf{C} \\
      \mathbf{0} & (\lambda\mathbf{I}-\mathbf{B})-\mathbf{C}^{T}(\lambda\mathbf{I}-\mathbf{A})^{-1}\mathbf{C} \\
       \end{array}\right)\right|     \nonumber \\
&=\left|\left(\begin{array}{cc}
      \lambda\mathbf{I}-\mathbf{A} & \mathbf{0} \\
\mathbf{0} & (\lambda\mathbf{I}-\mathbf{B})-\mathbf{C}^{T}(\lambda\mathbf{I}-\mathbf{A})^{-1}\mathbf{C} \\
   \end{array}\right)
\left(\begin{array}{cc}
     \mathbf{I} & -(\lambda\mathbf{I}-\mathbf{A})^{-1}\mathbf{C}  \\
 \mathbf{0} & \mathbf{I} \\
 \end{array}\right)\right|   \nonumber \\
&=\left|(\lambda\mathbf{I}-\mathbf{A})\left[(\lambda\mathbf{I}-\mathbf{B})-\mathbf{C}^{T}(\lambda\mathbf{I}
-\mathbf{A})^{-1}\mathbf{C}\right]\right|    \nonumber  \\
&=\left|\lambda^{2}\mathbf{I}-(\mathbf{A}+\mathbf{B})\lambda+
\left[\mathbf{A}\mathbf{B}-(\lambda\mathbf{I}-\mathbf{A})\mathbf{C}^{T}(\lambda\mathbf{I}-\mathbf{A})^{-1}\mathbf{C}
\right]\right|  .
\end{align}

The eigenvalues of matrix $\mathbf{C}$ are
\begin{equation*}
\gamma_{l}=\sum_{k=1}^{j}C_{1,k}e^{-2\pi \mathrm{i}(l-1)(k-1)/j}, \quad l=1,2,\cdots, j ,
\end{equation*}
where $C_{1,k}$ is the element in the first row and the $k$-th column of matrix $\mathbf{C}$.
The following equation can be obtained by Proposition \ref{circu},
\begin{equation*}
 \mathbf{U}^{-1}(\lambda\mathbf{I}-\mathbf{A})\mathbf{C}^{T}(\lambda\mathbf{I}-\mathbf{A})^{-1}
\mathbf{C}\mathbf{U}=\mathbf{diag}\{\bar{\gamma}_{l}\gamma_{l}\} ,
\end{equation*}
where $\bar{\gamma}$ represents the eigenvalues of matrix $\mathbf{C}^{T}$, and equals to the conjugate of $\gamma$.
So, characteristic equations can be obtained as below,
\begin{equation}\label{lambda}
 \lambda_{l}^{2}-(\alpha_{l}+\beta_{l})\lambda_{l}+\alpha_{l}\beta_{l}-\bar{\gamma}_{l}\gamma_{l}=0 ,
\quad 1 \leq l \leq j .
\end{equation}
One can also find that
\begin{align*}
\alpha_{l}=\alpha_{j+2-l}, \quad
\beta_{l} =\beta_{j+2-l}, \quad
\gamma_{l}\bar{\gamma}_{l}=\gamma_{j+2-l}\bar{\gamma}_{j+2-l} .
\end{align*}

If $\alpha_{l}$, $\beta_{l}$ and $\bar{\gamma}_{l}\gamma_{l}$ are calculated for $1\leq l \leq j$,
then the eigenvalues of the block circulant matrix $\mathbf{S}$
can be calculated by solving equations (\ref{lambda}).

\section{The set of masses for regular $n$-gon relative equilibria}

Consider a regular $n$-gon relative equilibrium.
There is only one set of identical infinitesimal masses for a regular $n$-gon
relative equilibrium when $n$ is odd, and there exist two kinds of infinitesimal bodies
arranged alternatively at the vertices of the regular $n$-gon when $n$ is even.
These can be shown to be true
by referring to a proposition given by Renner \& Sicardy (p.403) and the method in section 2.3.

Let's review that proposition in this paragraph.
With the help of equations (\ref{ki1}) and (\ref{Ftheta}), one can write conditional equations
for the relative equilibria of the $(1+n)$-body problem in the following way,
\begin{equation}\label{Fabs}
 \sum_{k=1,\neq i}^{n}\mu_{k}\mathcal{F}(\vartheta_{k}-\vartheta_{i})=0,
\quad i=1,2,\cdots,n .
\end{equation}
Suppose that angles $(\vartheta_{1}$, $\cdots$, $\vartheta_{n})$ correspond to a relative equilibrium.
Define $\mathcal{F}_{ik}\equiv \mathcal{F}(\vartheta_{k}-\vartheta_{i})$ with $k\neq i$,
$\mathcal{F}_{kk}=0$, and note that function $\mathcal{F}$ is odd. One can change
equations (\ref{Fabs}) into a matrix form,
\begin{equation}\label{matrixF}
 \left(\begin{array}{ccccc}
  0                & \mathcal{F}_{12} & \mathcal{F}_{13}   & \cdots & \mathcal{F}_{1n}     \\
 -\mathcal{F}_{12} & 0                & \mathcal{F}_{23}   & \cdots & \mathcal{F}_{2n}     \\
 \vdots            & \vdots           & \ddots             & \ddots & \vdots      \\
 -\mathcal{F}_{1n} & -\mathcal{F}_{2n} & -\mathcal{F}_{3n} & \cdots & 0
       \end{array} \right)
\left(\begin{array}{c}
       \mu_{1} \\     \mu_{2}  \\   \vdots \\   \mu_{n}
      \end{array} \right)
\equiv \mathbf{M}_{n}
\left(\begin{array}{c}
       \mu_{1} \\     \mu_{2}  \\   \vdots \\   \mu_{n}
      \end{array} \right)= \mathbf{0}_{n\times 1} .
\end{equation}
If $n$ is odd, and \textrm{rank}($\mathbf{M}_{n}$) is $n-\kappa$, where $\kappa$ is an odd integer,
then there exists a $\kappa$-parameter family of mass vectors $(\mu_{1}$, $\cdots$, $\mu_{n})$.
If $n$ is even, then \textrm{rank}($\mathbf{M}_{n}$) is generally $n$, and there is generally no family of mass vectors
for a relative equilibrium. However, Renner \& Sicardy just mentioned examples of when $n=2,3$
in their remarks.

Substitute the angle positions given in (\ref{angles})
into (\ref{matrixF}).
For a regular $n$-gon relative equilibrium with $n\geq 3$, we will show that,
\begin{prop}\label{pair}
When $n$ is odd, \emph{rank}\emph{(}$\mathbf{M}_{n}$\emph{)}$=n-1$, then there exists only one parameter
for the set of infinitesimal masses.
When $n$ is even, \emph{rank}\emph{(}$\mathbf{M}_{n}$\emph{)}$=n-2$,
then there exists a two-parameter family of mass vectors.
In addition, the two kinds of masses are alternatively arranged at the vertices.
\end{prop}

\begin{proof}
The rank of matrix $\mathbf{M}_{n}$ equals to the number of non-zero eigenvalues.
For a regular $n$-gon relative equilibrium solution, one finds that
\begin{align*}
 \mathcal{F}_{ij}&=\mathcal{F}(\vartheta_{j}-\vartheta_{i})=-\mathcal{F}(\vartheta_{i}-\vartheta_{j})=-\mathcal{F}_{ji} ,
\\
 \mathcal{F}_{ik}&=\mathcal{F}_{i+l \, \text{mod} \, n, \,  k+l \, \text{mod} \, n} .
\end{align*}
The two expressions above verify that the matrix $\mathbf{M}_{n}$ is both asymmetric and circulant.
Note that
\begin{displaymath}
 \mathcal{F}_{1k}+\mathcal{F}_{1,n+2-k}=\mathcal{F}\left( \frac{2\pi}{n}(k-1)\right)
+\mathcal{F}\left( \frac{2\pi}{n}(n+1-k)\right)=0 .
\end{displaymath}
The eigenvalues are
\begin{align*}
 \lambda_{l}&=\sum_{k=1}^{n}\mathcal{F}_{1k} \cdot e^{-2\pi \mathrm{i}(l-1)(k-1)/n} \nonumber \\
            &=\frac{1}{2}\sum_{k=2}^{n}\left[\mathcal{F}_{1k}\cdot e^{-2\pi \mathrm{i}(l-1)(k-1)/n}
            +\mathcal{F}_{1,n+2-k}\cdot e^{-2\pi \mathrm{i}(l-1)(n+1-k)/n}\right]
            \nonumber  \\
            &=\frac{1}{2}\sum_{k=2}^{n}\mathcal{F}_{1k}\cdot \left[ e^{-2\pi \mathrm{i}(l-1)(k-1)/n}
      -e^{-2\pi \mathrm{i}(l-1)(1-k)/n} \right]
             \nonumber  \\
            &=-\mathrm{i}\sum_{k=1}^{n-1}\mathcal{F}(\frac{2k\pi}{n})\sin\frac{2(l-1)k\pi}{n}
             \nonumber \\
           &=-\mathrm{i}
\biggl\{\frac{1}{2}\sum_{k=1}^{n-1}\left[\cos\frac{2lk\pi}{n}-\cos\frac{(4-2l)k\pi}{n}\right]
+\frac{1}{4}\sum_{k=1}^{n-1}\frac{\cos \frac{k\pi}{n}}{\sin^{2}\frac{k\pi}{n}}\sin\frac{2(l-1)k\pi}{n}
\biggl\}
\nonumber  \\
 &=-\mathrm{i}f_{1}(n,l) , \quad 1\leq l\leq n .
\end{align*}

One finds $\lambda_{1}\equiv 0$ and $\lambda_{n/2+1}=0$ when $n$ is even.
It is evident that $\lambda_{l}=-\lambda_{n+2-l}$ when $2\leq l\leq n$,
so one can just consider $2\leq l\leq \lfloor(n+1)/2\rfloor$. Here $\lfloor \cdot \rfloor$
represents the biggest integer that is no greater than the number inside the symbol.

The first part of $f_{1}(n,l)$ can be simplified as
\begin{displaymath}
 \frac{1}{2}\sum_{k=1}^{n-1}\left[\cos\frac{2lk\pi}{n}-\cos\frac{(4-2l)k\pi}{n}\right]
= \left\{ \begin{array}{ll}
     -\frac{n}{2},     & \textrm{if} \  l=2,    \\
     0,  &  \textrm{if} \  l=3,4,\cdots,\lfloor(n+1)/2\rfloor ,
   \end{array}\right.
\end{displaymath}
and the second part of $f_{1}(n,l)$ can be denoted as $f_{2}(n,l)$,
\begin{align*}
 f_{2}(n,l)=\frac{1}{4}\sum_{k=1}^{n-1}\frac{\cos \frac{k\pi}{n}}{\sin^{2}\frac{k\pi}{n}}\sin\frac{2(l-1)k\pi}{n},
\quad n\geq 3 .
\end{align*}

To check the convex property of the function $f_{2}(n,l)$, one may judge whether the
following expression is positive or negative.
\begin{equation*}
  f_{2}(n,l)+f_{2}(n,l+2)-2f_{2}(n,l+1)=-\sum_{k=1}^{n-1}\cos\frac{k\pi}{n}\sin\frac{2lk\pi}{n} ,
\end{equation*}
where $2\leq l< \lfloor (n-3)/2 \rfloor , \ n\geq 7$ .

One can find that the above equality is smaller than zero, because,
in $\sum_{k=1}^{n-1}\cos\frac{k\pi}{n}\sin\frac{2lk\pi}{n}$,
the positive part is greater than the absolute value of the negative part.
So function $f_{2}(n,l)$ is convex with the mouth downward.

The two ends of $f_{1}(n,l)$ can be shown to be positive when $n$ is large. When $n$ is odd, $\frac{n+1}{2}$
is an integer, and
 \begin{align*}
f_{1}(n,\frac{n+1}{2})&=\frac{1}{4}\sum_{k=1}^{n-1}\frac{\cos\frac{k\pi}{n}}{\sin^{2}\frac{k\pi}{n}}
\sin\frac{n-1}{n}k\pi
=\frac{1}{2}\sum_{k=1}^{(n-1)/2}\frac{(-1)^{k+1}}{\tan \frac{k\pi}{n}}>0 .
\end{align*}
When $n$ is even, $\lfloor\frac{n+1}{2}\rfloor=\frac{n}{2}$, and
\begin{align*}
 f_{1}(n,\frac{n}{2})&=\frac{1}{4}\sum_{k=1}^{n-1}\frac{\cos\frac{k\pi}{n}}{\sin^{2}\frac{k\pi}{n}}
\sin\frac{n-2}{n}k\pi=\sum_{k=1}^{n/2-1}\frac{\cos^{2}\frac{k\pi}{n}}{\sin\frac{k\pi}{n}}(-1)^{k+1}>0 .
\end{align*}
Thus $f_{1}(n,\lfloor(n+1)/2\rfloor)>0$ when $n\geq 3$.

On the other hand, we'll check whether $f_{1}(n,2)$ is greater than zero when $n$ is large.
\begin{align*}
 f_{1}(n,2)&=-\frac{n}{2}+\frac{1}{2}\sum_{k=1}^{n-1}\left(\frac{1}{\sin\frac{k\pi}{n}}-\sin\frac{k\pi}{n}\right)
\end{align*}
Apparently, $1/\sin x $ is convex with the mouth upward as $x\in (0,\pi)$, so one can have
the following relationship by the trapezoidal rule,
\begin{align}\label{1oversinx}
 \sum_{k=1}^{n-1}\frac{1}{\sin\frac{k\pi}{n}}\frac{\pi}{n}  & \geq
\left[\frac{1}{2}\frac{1}{\sin\frac{\pi}{n}}+\sum_{k=2}^{n-2} \frac{1}{\sin\frac{k\pi}{n}}
+\frac{1}{2}\frac{1}{\sin\frac{(n-1)\pi}{n}}\right]\cdot \frac{\pi}{n}
\nonumber \\
 & > \int_{\frac{\pi}{n}}^{\frac{(n-1)\pi}{n}}\frac{1}{\sin x}\mathrm{d}x
=2\ln \frac{\cos \frac{\pi}{2n}}{\sin \frac{\pi}{2n}} .
\end{align}
One can also get the summation of the series below,
\begin{eqnarray}\label{sumsin}
 \sum_{k=1}^{n-1}\sin \frac{k\pi}{n}=\frac{1+\cos\frac{\pi}{n}}{\sin\frac{\pi}{n}}
=\frac{\cos\frac{\pi}{2n}}{\sin\frac{\pi}{2n}} .
\end{eqnarray}
Define a function as below,
\begin{align*}
 f_{3}(x)&=-\frac{\pi}{4x}+\frac{1}{2x}\ln \frac{\cos x}{\sin x}-\frac{\cos x}{2\sin x}
\\
& > -\frac{\pi}{4x}+\frac{1}{2x}\ln \frac{\cos x}{\sin x}-\frac{1}{2\sin x}
\\
&=\frac{1}{2x}\left[\ln\frac{\cos x}{\sin x}-\frac{x}{\sin x}-\frac{\pi}{2}\right]
\end{align*}
The definition domain of $f_{3}(x)$ is $(0,\pi/2)$, and the zero root depends
on the function $f_{4}(x)$ as below,
\begin{equation*}
 f_{4}(x)=\ln\frac{\cos x}{\sin x}-\frac{2x}{\sin x}-\frac{\pi}{2} .
\end{equation*}
Let $x=\pi/2n$, and function $f_{4}(n)$ is a monotonously increasing function as $n\geq 3$.
One can find $f_{4}(41)<0$ and $f_{4}(42)>0$, then one has $f_{1}(n,2)>0$ if $n\geq 42$.

In addition, one can use one calculating software to calculate and
find that there are no very near zero values for $f_{1}(n,l)$ when $3\leq n\leq 41$ and
$2\leq l \leq \lfloor(n+1)/2 \rfloor$.
So, there is only one zero eigenvalue $\lambda_{1}$ if $n$ is odd, and
there are only two zero eigenvalues when $n$ is even. This completes the proof.
\end{proof}

\section{Linear stabilities of the regular $n$-gon relative equilibria when $n$ is even}

Exact expressions of the elements of the Hessian matrix
$\mathbf{V}_{\vartheta\vartheta}$ are given as below,
\begin{align*}
   V_{\vartheta_{i}\vartheta_{l}}&=-\mu_{i}\mu_{l}\left[\frac{3+\cos(\vartheta_{l}-\vartheta_{i})}{2r_{l,i}^{3}}
+\cos(\vartheta_{l}-\vartheta_{i})\right], \quad l\neq i ,   \\
V_{\vartheta_{i}\vartheta_{i}}&=-\sum_{l=1, l\neq i}^{2j}V_{\vartheta_{i}\vartheta_{l}}.
  \end{align*}
Especially, note that $V_{\vartheta_{i}\vartheta_{l}}=-\mu_{i}\mu_{l}\mathbf{f}(\vartheta_{l}-\vartheta_{i})$.
Consider that there are just two kinds of infinitesimal masses, let $\mu_{k}=\mu_{1}$ if $k$ is odd,
$\mu_{k}=\mu_{2}$ if $k$ is even.

$\mathbf{V}_{\vartheta\vartheta}$ is equivalent to the block circulant matrix
$\mathbf{S}$ in (\ref{block1}) by row and column exchanges.
The elements in the four blocks of $\mathbf{S}$ are
\begin{eqnarray*}
 A_{i,k}=V_{\vartheta_{2i-1}\vartheta_{2k-1}},\quad
 B_{i,k}=V_{\vartheta_{2i}\vartheta_{2k}}, \quad
 C_{i,k}=V_{\vartheta_{2i-1}\vartheta_{2k}} ,
\end{eqnarray*}
and $1\leq i, k\leq j$.
Matrix $\mathbf{S}$ is a symmetric matrix, so all its eigenvalues are real.
Also note that $\alpha_{l}$, $\beta_{l}$ and $\gamma_{l}\bar{\gamma}_{l}$ are real.
The aime of this section is to show that there exist no negative eigenvalues
in $\mathbf{S}$ such that the linear staibilites of the regular $n$-gon relative equilibria
can be assured by Theorem \ref{Moeckelthm}.

The sufficient conditions for all the roots of equations (\ref{lambda})
to be nonnegative are
\begin{eqnarray}\label{condition}
 \alpha_{l}+\beta_{l}\geq 0, \quad \alpha_{l}\beta_{l} \geq \bar{\gamma}_{l}\gamma_{l}=|\gamma|^{2}, \quad 1\leq l \leq j .
\end{eqnarray}
These conditions above will be checked by applying the Proposition \ref{circu}.

The expressions for $\alpha_{l}$, $\beta_{l}$, $\gamma_{l}$ and $\bar{\gamma}_{l}$
in (\ref{condition}) are given,
\begin{align*}
 \alpha_{l} &= V_{\vartheta_{1}\vartheta_{1}}+\sum_{k=2}^{j}
V_{\vartheta_{1}\vartheta_{2k-1}}e^{-2\pi \mathrm{i}(l-1)(k-1)/j},
\\
 \beta_{l}  &= V_{\vartheta_{2}\vartheta_{2}}+\sum_{k=2}^{j}V_{\vartheta_{2}\vartheta_{2k}}e^{-2\pi \mathrm{i}(l-1)(k-1)/j} ,
\end{align*}
\begin{eqnarray*}
 \gamma_{l}=\sum_{k=1}^{j}V_{\vartheta_{1}\vartheta_{2k}}e^{-2\pi \mathrm{i}(l-1)(k-1)/j}, \quad
   \bar{\gamma}_{l}=\sum_{k=1}^{j}V_{\vartheta_{2}\vartheta_{2k-1}}e^{-2\pi \mathrm{i}(l-1)(k-1)/j} .
\end{eqnarray*}
where
\begin{align*}
 V_{\vartheta_{1}\vartheta_{1}}&=\mu_{1}^{2}\sum_{k=1}^{j-1}\mathbf{f}\left(\frac{2k\pi}{j}\right)
+\mu_{1}\mu_{2}\sum_{k=1}^{j}\mathbf{f}\left(\frac{2k-1}{j}\pi \right),
\\
V_{\vartheta_{2}\vartheta_{2}}&=\mu_{2}^{2}\sum_{k=1}^{j-1}\mathbf{f}\left(\frac{2k\pi}{j}\right)
+\mu_{1}\mu_{2}\sum_{k=1}^{j}\mathbf{f}\left(\frac{2k-1}{j}\pi \right) .
\end{align*}
 Note that $\alpha_{l}$ and $\beta_{l}$ can be separated into two parts as described below,
\begin{align*}\label{alphabeta}
 \alpha_{l}&=\mu_{1}^{2}\mathbf{g}_{1}(j,l)+\mu_{1}\mu_{2}\mathbf{g}_{2}(j) ,
 \\
 \beta_{l}&=\mu_{2}^{2}\mathbf{g}_{1}(j,l)+\mu_{1}\mu_{2}\mathbf{g}_{2}(j) ,
\end{align*}
where
\begin{align*}
 \mathbf{g}_{1}(j,l)&=\sum_{k=1}^{j-1}\mathbf{f}\left(\frac{2k\pi}{j}\right)\left( 1-\cos\frac{2(l-1)k\pi}{j} \right) ,
 \\
\mathbf{g}_{2}(j)&=\sum_{k=1}^{j}\mathbf{f}\left( \frac{2k-1}{j}\pi \right) .
\end{align*}
 The exact forms of $\gamma_{l}$ and $\bar{\gamma}_{l}$ are also given,
\begin{align*}
 \gamma_{l}&=-\mu_{1}\mu_{2}\sum_{k=1}^{j}\mathbf{f}\left( \frac{2k-1}{j}\pi \right)
e^{-\frac{2\pi \mathrm{i}(l-1)(k-1)}{j}} ,
\\
\bar{\gamma}_{l}&=-\mu_{1}\mu_{2}\sum_{k=1}^{j}\mathbf{f}\left( \frac{2k-1}{j}\pi \right)
e^{-\frac{2\pi \mathrm{i}(l-1)k}{j}} .
\end{align*}

The expression of $|\gamma_{l}|$ can be calculated in the following way,
\begin{equation*}
 \gamma_{l}\bar{\gamma}_{l}=(\frac{\gamma_{l}+\bar{\gamma}_{l}}{2})^{2}
-(\frac{\gamma_{l}-\bar{\gamma}_{l}}{2})^{2} ,
\end{equation*}
and one gets
\begin{equation*}
 |\gamma_{l}|=\mu_{1}\mu_{2}\mathbf{g}_{3}(j,l) ,
\end{equation*}
where
\begin{align*}
\mathbf{g}_{3}(j,l)&=\sum_{k=1}^{j}\mathbf{f}\left( \frac{2k-1}{j}\pi \right)\cos\frac{(l-1)(2k-1)}{j}\pi .
\end{align*}
 The following expression of $\chi(j,l)$ is also given as it is necessary and crucial,
\begin{align}\label{chi}
\chi(j,l)&=\frac{1}{\mu_{1}\mu_{2}}(\alpha_{l}\beta_{l}-\gamma_{l}\bar{\gamma}_{l})
\nonumber \\
 &=\mu_{1}\mu_{2}\left[\mathbf{g}_{1}^{2}(j,l)
+\mathbf{g}_{2}^{2}(j)-\mathbf{g}_{3}^{2}(j,l)\right]
+(\mu_{1}^{2}+\mu_{2}^{2})\mathbf{g}_{1}(j,l)\mathbf{g}_{2}(j) .
\end{align}

Some propositions about $\mathbf{g}_{1}(j,l)$, $\mathbf{g}_{2}(j)$ and $\mathbf{g}_{3}(j,l)$
are given such that one can check the inequalities in (\ref{condition}).
\begin{prop}
 $\mathbf{g}_{1}(j,l)$ increases with $l$ when $2 \leq l \leq \lfloor (j+1)/2 \rfloor$ and $j\geq 3$,
where $\lfloor \cdot \rfloor$ means the maximum integer which is no greater than the real number inside.
\end{prop}
\begin{proof}
 \begin{align*}
  \mathbf{g}_{1}(j,l)&=\sum_{k=1}^{j-1}\left( \frac{3+\cos \frac{2k\pi}{j}}
{16\sin^{3}\frac{k\pi}{j}}+\cos\frac{2k\pi}{j}\right)
\left(1-\cos\frac{2(l-1)k\pi}{j} \right)   \nonumber  \\
  &= c_{1}(l)+c_{2}(l),
 \end{align*}
where
\begin{align*}
c_{1}(l)&=\sum_{k=1}^{j-1}\frac{3+\cos\frac{2k\pi}{j}}{16\sin^{3}\frac{k\pi}{j}}
\left(1-\cos\frac{2(l-1)k\pi}{j} \right), \nonumber  \\
 c_{2}(l)&=\sum_{k=1}^{j-1}\cos\frac{2k\pi}{j} \left(1-\cos\frac{2(l-1)k\pi}{j} \right)
\nonumber \\
&=\sum_{k=1}^{j}\cos\frac{2k\pi}{j} \left(1-\cos\frac{2(l-1)k\pi}{j} \right)
= -\frac{1}{2}\sum_{k=1}^{j}\cos\frac{2(l-2)k\pi}{j} .
\end{align*}
It is easy to find that $c_{2}(2)=-\frac{j}{2}$, and $c_{2}(l)=0$ if $l\neq 2$. Apparently
$c_{1}(l)$ is positive. In addition,
\begin{equation*}
 c_{1}(l+1)-c_{1}(l)=\sum_{k=1}^{j-1}
\frac{2-\sin^{2}\frac{k\pi}{j}}{4\sin^{2}\frac{k\pi}{j}}\sin\frac{(2l-1)k\pi}{j}>0 ,
\end{equation*}
so $c_{1}(l)$ increases with $l$ when $2\leq l\leq \lfloor(j+1)/2\rfloor$.
 The conclusion is thus apparent.
\end{proof}
\begin{prop}
 For all $j\geq 7$, we have $\mathbf{g}_{1}(j,2)>0$.
\end{prop}
\begin{proof}
 \begin{align*}
 \mathbf{g}_{1}(j,2) & =-\frac{j}{2}+\sum_{k=1}^{j-1}\left(\frac{1}{2\sin \frac{k\pi}{j}}
-\frac{1}{4}\sin \frac{k\pi}{j}\right) .
\end{align*}
One refers to the relationships of (\ref{1oversinx}) and (\ref{sumsin}),
and finds
\begin{align*}
 \mathbf{g}_{1}(j,2)>-\frac{j}{2}+\frac{j}{\pi}\ln\frac{\cos\frac{\pi}{2j}}{\sin\frac{\pi}{2j}}
-\frac{1}{4}\frac{\cos \frac{\pi}{2j}}{\sin \frac{\pi}{2j}} .
\end{align*}
Let $x=\frac{\pi}{2j}$, and one has
\begin{align*}
  h_{1}(x)&=\frac{1}{2x}\left(\ln\frac{\cos x}{\sin x}-\frac{x\cos x}{2\sin x}
-\frac{\pi}{2}\right) \nonumber \\
&> \frac{1}{2x}\left(\ln\frac{\cos x}{\sin x}-\frac{x}{2\sin x}
-\frac{\pi}{2}\right), \nonumber  \\
h_{2}(x)&=\ln\frac{\cos x}{\sin x}-\frac{x}{2\sin x}-\frac{\pi}{2} .
\end{align*}
$h_{2}(x)$ is a monotonously decreasing function of $x\in (0,\pi/2)$, so $\mathbf{g}_{1}(j,2)$ is an increasing function
of $j\geq 2$. One finds $\mathbf{g}_{1}(6,2)<0$, and  $\mathbf{g}_{1}(7,2)>0$. So
$\mathbf{g}_{1}(j,2)>0$ for all $j\geq 7$.
\end{proof}
\begin{prop}
  For all $j\geq 1$, $\mathbf{g}_{2}(j)>0$.
\end{prop}
\begin{proof}
As
\begin{equation*}
 \mathbf{g}_{2}(j)=\sum_{k=1}^{j}\frac{1}{4\sin \frac{2k-1}{2j}\pi}
(\frac{1}{\sin^{2}\frac{2k-1}{2j}\pi}-\frac{1}{2})>0, \quad \text{if} \quad j\geq 1 .
\end{equation*}
\end{proof}

\begin{prop} \label{Propg2g3}
  For all $j\geq 2$, $\mathbf{g}_{2}^{2}(j)-\mathbf{g}_{3}^{2}(j,2)>0$.
\end{prop}
\begin{proof}
 One has
\begin{displaymath}
 \mathbf{g}_{2}^{2}(j)-\mathbf{g}_{3}^{2}(j,2)=\left[\mathbf{g}_{2}(j)+\mathbf{g}_{3}(j,2)\right]\cdot
\left[\mathbf{g}_{2}(j)-\mathbf{g}_{3}(j,2)\right] ,
\end{displaymath}
\begin{align*}
 \mathbf{g}_{2}(j)+\mathbf{g}_{3}(j,2)&=\sum_{k=1}^{j}\mathbf{f}\left( \frac{2k-1}{j}\pi \right)
\cdot \left( 1+\cos \frac{2k-1}{j}\pi \right)
\nonumber   \\
&=\frac{j}{2}+\sum_{k=1}^{j}\left[\left( \frac{1}{\sin^{2}\frac{2k-1}{2j}\pi}-1 \right) \left(
\frac{1}{2\sin\frac{2k-1}{2j}\pi}-\frac{\sin\frac{2k-1}{2j}\pi}{4}\right)\right]>0 ,
\end{align*}
and
\begin{align*}
 \mathbf{g}_{2}(j)-\mathbf{g}_{3}(j,2)&=\sum_{k=1}^{j}\mathbf{f}\left( \frac{2k-1}{j}\pi \right)
\left( 1-\cos \frac{2k-1}{j}\pi \right)
\nonumber   \\
&=\sum_{k=1}^{j}\left(\frac{1}{2\sin\frac{2k-1}{2j}\pi}-\frac{\sin\frac{2k-1}{2j}\pi}{4}-\frac{1}{2}\right)
\nonumber \\
&>\frac{j}{\pi}\ln\frac{\cos\frac{\pi}{4j}}{\sin\frac{\pi}{4j}}
-\frac{1}{4\sin\frac{\pi}{2j}}-\frac{j}{2} .
\end{align*}
The function
\begin{equation*}
 h_{3}(x)=\frac{1}{2x}\left( \ln\frac{\cos\frac{x}{2}}{\sin\frac{x}{2}}-\frac{x}{2 \sin x}-\frac{\pi}{2}\right)
\end{equation*}
monotonously decreases as $x\in(0,\pi)$.
So $\mathbf{g}_{2}(j)-\mathbf{g}_{3}(j,2)$ increases with $j$ if $x=\frac{\pi}{2j}$.
As $\mathbf{g}_{2}(2)-\mathbf{g}_{3}(2,2)=\frac{3\sqrt{2}}{4}-1>0$. So,
$ \mathbf{g}_{2}(j) \pm\mathbf{g}_{3}(j,2)>0$
when $j\geq 2$.
\end{proof}

As a result of the above propositions, $\chi(j,2)>0$ is shown when $j\geq 7$.
Note that when $l=1$, $\mathbf{g}_{1}(j,1)\equiv0$,
but one can still find that $\chi(j,1)>0$ because of Proposition \ref{Propg2g3}.
Also note that when $2\leq l \leq \lfloor (j+1)/2 \rfloor$ and $j\geq 3$, $\mathbf{g}_{1}(j,l)$
increases with $l$, so $\chi(j,l)$ increases with $l$ as well. In all,
one has $\chi(j,l)>0$ when $j\geq 7$.

If $2\leq j\leq 6$ and one still looks for linear stability conditions,
one just needs to check whether $\chi(j,2)>0$ and $\alpha_{2}+\beta_{2}>0$. In these cases,
one gets $\mathbf{g}_{1}(j,2)<0$, and the discriminant of $\chi(j,2)$ is
\begin{equation}\label{h4}
 h_{4}=\left[\mathbf{g}_{1}^{2}(j,2)+\mathbf{g}_{2}^{2}(j)
-\mathbf{g}_{3}^{2}(j,2)\right]^{2}-4 \left[ \mathbf{g}_{1}(j,2)\mathbf{g}_{2}(j) \right]^{2} .
\end{equation}
$h_{4}$ is calculated to be positive when $4\leq j\leq 6$, and negative when $2\leq j\leq 3$.
So when $2\leq j\leq 3$, $\chi(j,2)<0$.

When $4\leq j\leq 6$, the intervals of $\mu_{1}/\mu_{2}$ can be obtained
by calculating the roots of the equation $\chi(j,2)=0$ easily. Such intervals can ensure the
correctness of the inequalities in (\ref{condition}). These conclusions are summarized as below,
\begin{thm}\label{1+2j}
Given a regular $n$-gon relative equilibrium of the $(1+n)$-body
problem with $n=2j$ and $j\geq 2$. If there are two kinds of
infinitesimal masses $\mu_{1}$ and $\mu_{2}$, the two kinds of
masses should be distributed alternatively. It is shown that such a
configuration is linearly stable if and only if $j\geq 7$, or one of
the following conditions should be satisfied,
\begin{eqnarray}\label{muJie}
 \frac{\mu_{1}}{\mu_{2}}\in \left( \frac{-h_{5}+\sqrt{h_{4}}}{2\mathbf{g}_{1}(j,2)\mathbf{g}_{2}(j)},
\frac{-h_{5}-\sqrt{h_{4}}}{2\mathbf{g}_{1}(j,2)\mathbf{g}_{2}(j)} \right), \quad j=4,5,6 ,
\end{eqnarray}
where $h_{4}$ is given in \emph{(}\ref{h4}\emph{)} and
\begin{displaymath}
 h_{5}=\mathbf{g}_{1}^{2}(j,2)+\mathbf{g}_{2}^{2}(j)-\mathbf{g}_{3}^{2}(j,2).
\end{displaymath}
 The intervals in \emph{(}\ref{muJie}\emph{)} can be calculated as
\begin{displaymath}
 \frac{\mu_{1}}{\mu_{2}} \in\begin{cases}
   (0.39601454048825, 2.525159805412902),  &   \text{if} \quad j=4,\\
   (0.16709497914366, 5.984620274797297),  &   \text{if} \quad j=5,\\
   (0.061964963348688, 16.13815204525851), &   \text{if} \quad j=6.
                 \end{cases}
\end{displaymath}
\end{thm}

\section{Conclusion and Discussion}

We've considered the linear stabilities of the regular $n$-gon relative equilibria of the $(1+n)$-body problem.
When $n$ is odd, we've shown that there is only one kind of infinitesimal mass.
We recall the conclusion of Moeckel(1994) that the regular $n$-gon relative equilibria
with all infinitesimal masses equal are linearly stable if and only if $n\geq 7$. We have also shown that,
if $n$ is even, there can exist two kinds of infinitesimal bodies arranged alternatively at the vertices of the
regular $n$-gon. Such relative equilibria with $n$ even are linearly stable if and only if $n\geq 14$, or
the ratios of the two kinds of masses must be in given intervals for $n=8, 10, 12$.
See the Proposition \ref{pair} and the theorem \ref{1+2j}.

The allure of the relative equilibria of the $(1+n)$-body problem comes from
both mathematical and astronomical interests. The features of stationary configurations of the $(1+n)$-body problem
are fascinating, but it is difficult to explain those features.
One question is what they are like and how many there are for $n$ given and not necessarily
all equal infinitesimal masses. Only special cases can be analytically solved for this question,
as the stationary solutions are generally difficult to calculate. The second one is the inverse one,
which is to look for a set of masses for $n$ given positions on a circle.

There are also some questions about the stability. The effects of eccentricities of the orbits
of the infinitesimal bodies and the oblateness of the primary can also be considered on the
linear stability of the relative equilibria of the $(1+n)$-body problem.

In addition, as a result of the relationship between the relative equilibria of the $(1+n)$-body problem
and those of the $N$-body problem with $N=n+1$, we can consider
the existence and stability of some periodic solutions near the relative equilibria of the $N$-body problem
with two kinds of small masses, by referring to Meyer \& Schmidt (1993)\cite{Meyer} and Pascual (1998)\cite{Pascual}.
The nonlinear stability of the relative equilibria or periodic perturbations
of the $N$-body problem with $N=n+1$ is also deserved to be researched, and one can try to
analyze the properties of higher orders or do some simulations by using a symplectic integrator.


\subsection*{Acknowledgment}
Many thanks to Prof. Alain Albouy in IMCCE, for his great help on this paper.
Thanks to the National Natural Science Foundation of China under grant No.10833001.
I also want to express my gratitude to the support of the Joint Doctoral Promotion program between CNRS and CAS.
\end{document}